\documentclass[11pt, a4paper]{amsart}
\usepackage{amssymb}

\usepackage[a4paper]{geometry}
\usepackage{hyperref}
\usepackage{graphicx,amssymb,amsfonts,epsfig,amsthm,a4,amsmath,url}
\usepackage[utf8]{inputenc}
\usepackage{tikz}
\usepackage{textcomp}
\usepackage{comment}
\usetikzlibrary{matrix}
\usepackage{xcolor}
\usepackage[all]{xy}
\usepackage{stmaryrd}

\newtheorem*{que}{Question}

\newtheorem{ThmIntro}{Theorem}

\newtheorem{PropIntro}[ThmIntro]{Proposition}

\newtheorem{thm}{Theorem}[section]
\newtheorem{cor}[thm]{Corollary}
\newtheorem{lem}[thm]{Lemma}

\newtheorem{prop}[thm]{Proposition}
\newtheorem*{prop3}{Proposition 3}
\theoremstyle{definition}
\newtheorem{defn}[thm]{Definition}

\newtheorem{fac}[thm]{Fact}

\newtheorem*{Ack}{Acknowledgements}
\theoremstyle{remark}
\newtheorem{rem}[thm]{Remark}
\newtheorem{rems}[thm]{Remarks}

\newtheorem{exs}[thm]{Examples}
\newtheorem*{Orga}{Organization}

\numberwithin{equation}{section}

\newcommand{\Z}{\mathbb{Z}}

\newcommand{\N}{\mathbb{N}}
\newcommand{\R}{\mathbb{R}}

\newcommand{\Q}{\mathbb{Q}}

\newcommand{\PP}{\mathcal{P}}
\newcommand{\p}{\mathbb{P}}

\DeclareMathOperator{\trdeg}{trdeg}
\DeclareMathOperator{\dev}{dev}
\DeclareMathOperator{\krull}{Krull}
\DeclareMathOperator{\Ann}{Ann}
\newcommand{\exseq}[3]{#1 \hookrightarrow #2 \twoheadrightarrow #3}

\newcommand{\supp}{\textnormal{Supp}}

\newcommand{\GGG}{\Gamma}
\newcommand{\LLL}{\Lambda}
\newcommand{\OO}{\Omega}

\begin{document}
\title{Metabelian groups with large return probability}
\author{Lison Jacoboni}
\address{Laboratoire de mathématiques d'Orsay, Univ. Paris-Sud, CNRS, Université Paris Saclay, 91405 Orsay, France}
\email{lison.jacoboni@math.u-psud.fr}
\date{\today}
\thanks{The author was supported by ANR-14-CE25-0004 GAMME}

\maketitle

\begin{abstract}
We study the return probability of finitely generated metabelian groups. We give a characterization of metabelian groups whose return probability is equivalent to $\exp(-n^\frac1{3} )$ in purely algebraic terms, namely the Krull dimension of the group. Along the way, we give lower bounds on the return probability for metabelian groups with torsion derived subgroup, according to the dimension. We also establish a variation of the famous embedding theorem of Kaloujinine and Krasner for metabelian groups that respects the Krull dimension. 
\end{abstract}

\tableofcontents

\section{Introduction}

Let $G$ be a countable group. For any probability measure $\mu$ on $G$, one can consider the \emph{random walk on $G$ driven by $\mu$}: this is a sequence $(S_i)_{i \geq 0}$ of random variables valued in $G$ such that $S_0 = X_0$ and $S_{n + 1} = S_n X_{n + 1}$, where $(X_i)_{i \geq 0}$ are independent and identically distributed random variables with probability distribution $\mu$.
If the current state is $x$, the probability of being in $y$ at the next step is $\mu(x^{-1}y)$. This implicitly defines a probability measure $\p_\mu$ on $G^\N$ such that
\begin{equation*}
\p_\mu(S_n = y \mid X_0 = x) = \mu^{(n)}(x^{-1}y),
\end{equation*} 
where $\mu^{(n)}$ denotes the $n$-fold convolution of $\mu$. Recall that if $f$ and $h$ are two functions on $G$, the convolution of $f$ and $h$ is $f * h(g) = \sum_k f(k) h(k^{-1}g)$.

An interesting class of examples arises from a probability measure whose support generates the whole group $G$ as a semigroup. Within this class, a fundamental example is given by a finitely generated group $G$ with $\mu = \mu_S$, the uniform distribution on a finite and symmetric generating set $S$. Such a probability measure satisfies $ \mu_S(g)= \mu_S(g^{-1})$, for all $g \in G$. In general, one says that a probability measure $\mu$ is \emph{symmetric} whenever  this latter relation is true. 

Let $p_{2n}^{\mu, G} = \p_\mu(X_{2n} = e \mid X_0 = e)$ denote the probability of  return in $2n$ steps to the origin $e$ of $G$ (even times $2n$ are considered to avoid parity issues: namely, the simple random walk on $\Z$, with usual generating set, cannot reach $0$ at odd times). 

If $\varphi, \psi$ denote two monotone functions, we use the notation $\varphi \precsim \psi$ if there exist positive constants $c$ and $C$ such that $c\varphi(Ct) \leq \psi(t)$ (possibly for $t$ in $\N$ if $\N$-valued functions are considered). If the symmetric relation $\varphi \succsim \psi$ also holds, we write $\varphi \sim \psi$ and say that $\varphi$ and $\psi$ have the same asymptotic behaviour. This is an equivalence relation. 

A theorem of Pittet and Saloff-Coste \cite{PSC00} asserts that any two symmetric and finitely supported probability measures with generating support give rise to equivalent return probabilities. Let $p_{2n}^G$ denote this invariant, dropping the $G$ whenever the group is clear from the context. 
Also, if $\mu$ is a symmetric probability measure with generating support and finite second moment, that is $\sum_g \lvert g \rvert^2 \mu(g) < \infty$, then $p_{2n}^{\mu, G}$ belongs to the class of $p_{2n}^G$.  

\bigskip

Understanding how the random walk behaves allows to have insight into the large-scale geometry of the group. We give below a brief picture of what is known about $p_{2n}^G$, more details are to follow in part \ref{ssec-return proba}.

\medskip

In his thesis, Kesten proved that non-amenable groups are characterized by the fact that they behave the worst, for their return probability decays exponentially fast (\cite{Kesten59}). On the other hand, a nilpotent group, which necessarily has polynomial growth of some definite degree $d$ behaves like $\Z^d$: this follows from Varopoulos (\cite{Var85},\cite{Var92}).

What about amenable groups of exponential growth? Hebisch and Saloff-Coste (\cite{HSC}) proved that if $G$ has exponential growth, then 
\begin{equation}\label{eq-HSC}
p_{2n}^G \precsim \exp(-n^\frac{1}{3}).
\end{equation}

Moreover, if $G$ is a discrete subgroup of a connected Lie group, then $G$ is amenable of exponential growth if and only if $p_{2n}^G \sim \exp(-n^\frac{1}{3})$.

\medskip

However, the world of amenable groups of exponential growth can be very wild, and many other behaviours can occur. As an example, for any nontrivial finite group $F$, the return probability of the wreath product $F \wr \Z^d$ is  
\begin{equation*}
p_{2n}^{F \wr \Z^d} \sim \exp(-n^\frac{d}{d+2})
\end{equation*}
(\cite{Erschler06}).

\medskip 

It appears that, for every small positive $\epsilon$, there exist $1 - \epsilon < \alpha < 1$, and a solvable group whose return probability is equivalent to $\exp(-n^\alpha)$ (see also \cite{PSC02}). These solvable groups can be chosen to be $2$-step solvable (in other words \emph{metabelian} or equivalently having abelian derived subgroup) and hence
are relatively tame within the class of solvable groups of arbitrary derived length.
Comparing with the situation for discrete subgroups of Lie groups, the following question arises naturally.

\begin{que}
Among finitely generated solvable groups, is it possible to characterize the groups whose return probability satisfies $p_{2n} \succsim \exp(-n^\frac{1}{3})$ ?
\end{que}

So far, \ref{eq-HSC} is known to be sharp for lamplighter group $F \wr \Z$, with $F$ a nontrivial finite group (\cite{Var83}), for polycyclic groups (\cite{Alex}), for solvable Baumslag--Solitar (\cite{CGP01}), and for solvable groups of finite Prüfer rank,  (\cite{PSC00},\cite{KL}). A group has \emph{finite Prüfer rank} if there exists an integer $r$ such that any finitely generated subgroup can be generated by at most $r$ elements and the least such $r$ is the rank. 
Lastly, Tessera proved in \cite{T13} the reverse inequality for a class of groups containing all discrete subgroups of solvable algebraic groups over a local field. This class contains in particular lamplighter groups and torsion-free solvable groups of finite Prüfer rank.

\bigskip

Here, we answer the above question in the case of metabelian groups, showing a connexion between the Krull dimension of a metabelian group and the asymptotic behaviour of its return probability. 

\medskip

A metabelian group $G$ is an extension of a abelian group by another abelian group, namely
\begin{equation*}
[G, G] \hookrightarrow G \twoheadrightarrow G_{ab},
\end{equation*}
where we denote by $[G, G]$ the derived subgroup of $G$ and by $G_{ab}$ its abelianization $G/[G, G]$.
The subgroup $[G, G]$ carries a natural structure of $\Z G_{ab}$-module, coming from the action by conjugation.

\medskip

The notion of Krull dimension, introduced in \cite{GR67}, plays an important role in the theory of rings and modules. In \cite{Tus03}, Tushev extended this notion to groups. We recall the classical treatment of Krull dimension in non-commutative ring theory and its generalization to groups in section \ref{sec-Kd}, after which we study it for metabelian groups.

For the moment, it is enough to mention that the Krull dimension of a metabelian group $G$, denoted $\krull(G)$, is characterized by the following two properties (see Proposition \ref{prop-Kd meta}): 
\begin{itemize}
\item If the Krull dimension of the $\Z G_{ab}$-module $[G, G]$ is positive, then they are equal: 
\begin{equation*}
\krull(G) = \krull_{\Z G_{ab}}([G, G]).
\end{equation*}
\item otherwise, when the module $[G, G]$ has Krull dimension $0$ (ie is finite), $\krull(G)$ is defined to be $1$ if $G$ is infinite, $0$ if it is finite.
\end{itemize}

\begin{ThmIntro}\label{thm-characterization}
Let $G$ be a finitely generated metabelian group of Krull dimension $k$. Then,
\begin{equation*}
k \leq 1 ~~ \Leftrightarrow ~~ p_{2n}^G \succsim \exp(-n^\frac1{3}).
\end{equation*}
\end{ThmIntro}

Together with (\ref{eq-HSC}), this implies that the finitely generated metabelian groups with exponential growth and large return probability, that is $p_{2n}^G \sim \exp(-n^\frac1{3})$, are exactly those of dimension $0$ or $1$.

This is a consequence of studying the impact of Krull dimension on the structure of the group, which yields the more precise lower bound:

\begin{ThmIntro}\label{thm-lower bound torsion}
Let $G$ be a metabelian group of Krull dimension $k$, $k \geq 1$. Assume that $[G, G]$ is torsion. Then,
\begin{equation*}
p_{2n}^G \succsim \exp(-n^\frac{k}{k+2}).
\end{equation*}
\end{ThmIntro}

A metabelian group $G$ is \emph{split} if it admit an exact sequence
\begin{align*}
\exseq{M}{G}{Q},
\end{align*}
that splits, where $M$ and $Q$ are abelian groups. The group $G$ is therefore isomorphic to the semi-direct product $M \rtimes Q$.

\begin{PropIntro}\label{prop-split torsion}
Let $G$ be a finitely generated split metabelian group of Krull dimension $k$. If $[G, G]$ is torsion, then
\begin{equation*}
p_{2n}^G \sim \exp(-n^\frac{k}{k+2}).
\end{equation*} 
\end{PropIntro}

To obtain these estimates, we reduce to a more tractable case. This is the object of the following embedding theorem. A more detailed statement can be find in subsection \ref{ssec-embedding}.

\begin{ThmIntro}\label{thm-kkv}
Every finitely generated metabelian group, which is the extension of an abelian group by a finitely generated abelian group $Q$, can be embedded inside a finitely generated split metabelian group $B \rtimes Q$, of the same Krull dimension, with $B$ abelian.
\end{ThmIntro}

The other direction in the proof of Theorem \ref{thm-characterization} is a consequence of the existence of some special subgroups.

\begin{PropIntro}\label{prop-special}
Let $G$ be a metabelian group of Krull dimension $k$. 

If $k \geq 2$, then $G$ has a subgroup isomorphic to either $\Z \wr \Z$, or to $B_2^{(p)}$ for some prime $p$.
\end{PropIntro}

\begin{Orga} In section \ref{sec-Kd}, we recall the definition and some useful properties of the Krull dimension of a module or a group. Then, we establish basic facts in the metabelian case. Section \ref{sec-prelim} contains preliminary results: the adaptation to the torsion case of computations by Saloff-Coste and Zheng (\cite{SCZ14}) of the return probability of the free metabelian group, and the proof of Theorem \ref{thm-kkv}, which is a variation of an embedding theorem of Kaloujinine and Krasner. Proposition \ref{prop-special} is proved in section \ref{sec-specsub}. Finally, in section \ref{sec-folner}, we study sequences of F\o lner couples in order to apply the machinery of \cite{CGP01} for lower bounds on the return probability: we state in particular that such sequences go to a quotient and explain how to construct them for specific extensions.
\end{Orga}

\begin{Ack}
This paper contains some of the results of my doctoral dissertation. The question of links between the return probability of a metabelian group and the Krull dimension as a module of the derived subgroup of a metabelian group was suggested by Yves Cornulier and Romain Tessera. 
I am very grateful to them for many valuable discussions,  helpful suggestions and careful reading. 
I would also like to warmly thank Peter Kropholler for stimulating discussions, comments and careful reading, as well as Anna Erschler for her interest and an interesting conversation.
\end{Ack}

\section{Krull dimension}\label{sec-Kd}

Before defining the Krull dimension of a metabelian group, we recall the definition for a module. In this paper, all rings considered will be \textbf{commutative with one}.

\subsection{Krull dimension of a module}

We mention two equivalent definitions: the first one will be generalized later to define the Krull dimension of a group, while the second will be sometimes easier to handle.

\subsubsection{Krull dimension interpreted as the deviation of a poset}\label{subsubsec-Kd}

We follow \cite{MCR}.

\medskip

Let $A$ be a poset. If $a \leq b$, let $[a,b]  = \{ x \in A \mid a \leq x \leq b \}$. This is a subposet of $A$, called \emph{interval}, or \emph{factor}, from $a$ to $b$. A \emph{descending chain} is a chain $(a_i)_i$ of elements of $A$ such that $a_1 \geq a_2 \geq \dots$, and the intervals $[a_{i+1}, a_i]$ are the \emph{factors of the chain}. If every such descending chain is eventually constant, we say that $A$ satisfy the \emph{descending chain condition.}
Similarly, one defines the \emph{ascending chain condition}. A poset $A$ is \emph{trivial} if $a \leq b$ implies $a = b$, for all $a, b \in A$.

\begin{defn}
The \emph{deviation} of $A$, denoted $\dev A$, if it exists, is 
\begin{itemize}
\item $- \infty$, if $A$ is empty or trivial,
\item $0$, if $A$ is non-trivial and satisfies the descending chain condition,
\item and in general by induction: $\dev A$ is defined and equal to an ordinal $n$, provided $\dev A$ is not equal to $m$ for every $m < n$, and in any descending chain of $A$, all but finitely many factors have deviation defined and less than $n$.
\end{itemize}
\end{defn} 

Note that a poset may not have a deviation. A sufficient condition is

\begin{prop}[\cite{MCR}]\label{prop-existence of Kd}
Any poset with the ascending chain condition has a deviation.
\end{prop}

We may now define the Krull dimension of a module as the deviation of a natural associated poset.

\begin{defn}\label{defn-Kd1}
Let $M$ be a $R$-module. Denote by $\mathcal L_R(M)$ the poset of $R$-submodules of $M$.  
The \emph{Krull dimension of $M$ as a $R$-module} is defined as $$\krull_R (M) = \dev \mathcal L_R(M),$$ whenever it exists. Otherwise, $M$ \emph{does not have a Krull dimension}. 
We may forget the reference to the ring whenever it is clear from the context.
\end{defn}

\smallskip

\begin{rems}~
\begin{enumerate}
\item If $A$ is a ring, the Krull dimension of $A$ denotes the Krull dimension of the $A$-module $A$, written $\krull(A)$.
\item A module satisfying the descending chain condition is called \emph{artinian}. Hence, modules of dimension $0$ are just artinian modules.
\end{enumerate}
\end{rems}

\begin{lem}[\cite{MCR}]\label{prop-kdim mod max}
If $N$ is a submodule of $M$, then
\begin{equation*}
\krull (M) = \max \{ \krull (N), \krull (M/N) \}.
\end{equation*}
\end{lem}

\smallskip

\subsubsection{Krull dimension interpreted as the dimension of a faithful ring}~

For commutative Noetherian rings, there is an equivalent way of defining of the Krull dimension in terms of length of chains of prime ideals. One can then derive an equivalent definition for the Krull dimension of a module over a commutative Noetherian ring. These equivalent definitions will sometimes turn out to be more tractable. References for this are \cite{eisenbud} and \cite{MCR}.

\begin{prop}[\cite{MCR} , Theorem 4.8]
Let $A$ be a commutative Noetherian ring. The Krull dimension of $A$, when finite, is equal to the supremum among the $r$ such that there exists a chain
\[ \PP_0 \subsetneq \PP_1 \subsetneq \dots \subsetneq \PP_r \]
of prime ideals in $A$. 
\end{prop}

\begin{rem}
For simplicity, the equivalent definition just stated is for \emph{finite} Krull dimension.
It is possible to refine it using induction so as to obtain an ordinal and get a general equivalent definition for Krull dimension (see \cite{MCR} 6.4).
\end{rem}

\smallskip

\begin{exs}\label{ex-krullpol}~
\begin{enumerate}
\item Any field has dimension $0$, and any principal ideal domain that is not a field has dimension $1$.
\item Polynomial rings over a Noetherian ring $A$  satisfy: 
\[ \krull(A[X_1, \dots, X_n]) = \krull(A) + n. \]
Considering invertible variables does not change the dimension:
\begin{equation*}
\krull(A[X_1^{\pm 1}, \dots, X_n^{\pm 1}]) = \krull(A) + n. 
\end{equation*}
\end{enumerate}
\end{exs}

Given fields $K\subset L$, the \emph{degree of transcendence} of $L$ over $K$ will be denoted by $\trdeg_K L$. For any module $M$ and any element $c$ in $M$, we denote by $M_c$ the localization of $M$ with respect to $\{c^n ; n \in \N\}$.

The second part of example \ref{ex-krullpol} actually generalizes : there is a connexion between Krull dimension and degree of transcendence. 
Before stating it, we recall Noether's normalization theorem (\cite{Noether1926}, see \cite{bosch} for this generalized version).

\begin{thm}[Noether's normalization theorem]\label{thm-noether}~

Let $A$ be an integral domain and let $A \subset R$ be an extension of $A$ such that $R$ is finitely generated as an $A$-algebra.

Then, there is a non-zero element $c$ in $A$ and elements $z_1, \dots, z_k$ in $R_c$ algebraically independent over $A_c$ such that $R_c$ is a finitely generated $A_c[z_1, \dots, z_k]$-module. 
\end{thm}

\begin{rem}
The subring $A_c[z_1, \dots, z_k]$ is isomorphic to a polynomial ring over $A_c$. Note that $k$ may be zero.

Moreover, if $A = K$ is a field, then $K_c = K$ for every non-zero element $c$  and the theorem says that every finitely generated $K$-algebra is a finitely generated module over a polynomial ring over $K$. The number $k$ of indeterminates is the Krull dimension of $R$. 
\end{rem}

Krull dimension behaves well with respect to extension, as stated in the next proposition.

\begin{prop}[\cite{eisenbud}]\label{prop-Kd trdeg}~
Let $A$ be an integral domain, $K$ its fraction field and $\mathbb P$ the prime subfield of $K$.
\begin{enumerate}
\item If $A$ is finitely generated as a $\mathbb P$-algebra, then,
\begin{equation*}
\krull (A ) = \trdeg_\mathbb P K.
\end{equation*}
\item If $A$ is finitely generated as a $\Z$-algebra, then
\begin{equation*}
\krull (A)  = \trdeg_\Q K + 1.
\end{equation*}
\end{enumerate}
\end{prop}

This point of view gives another way to compute the Krull dimension of a module, as the dimension of a faithful ring.

\begin{prop}[\cite{MCR}]
Let $A$ be a commutative Noetherian ring and $M$ a finitely generated $A$-module.
Then, 
\begin{equation}\label{eq-Kd2}
\krull_A(M) = \krull \left(A/\Ann(M) \right),
\end{equation}
where $\Ann(M)$ denotes the annihilator of $M$, that is $\Ann(M) = \{a \in A \mid aM = 0 \}.$
\end{prop}

\subsubsection{Refinements}
It is useful to distinguish the contributions to Krull dimension that are made
by torsion modules and by torsion-free modules. 
Let $M$ be a module and $T(M)$ be the torsion subgroup of the group $M$. The group $T(M)$ also carries a $\Z Q$-module structure, hence has finite exponent, and is a direct factor. 

\begin{lem}[\cite{Hall59}]\label{lem-dec of M}
Let $M$ be a Noetherian module and  denote by $T(M)$ the torsion subgroup of $M$. Then,  $T(M)$ is a submodule of $M$ and there is a torsion-free subgroup $N$ of $M$ such that, as groups , $M = T(M) \oplus N$. 

In particular, the group $N$ is isomorphic to $M/T(M)$.
\end{lem}

\begin{defn}
The \emph{torsion-free Krull dimension} of $M$ is $\krull^0(M) = \krull(M/T(M))$ and the \emph{torsion Krull dimension} of $M$ is $\krull^t(M) = \krull(T(M))$.
\end{defn}

\begin{rem}
Proposition \ref{prop-kdim mod max} implies that
$\krull(M)  = \max\{\krull^0(M), \krull^t(M)\}.$
\end{rem}

\smallskip

\subsubsection{Krull dimension and associated prime ideals}~

We end this subsection collecting some lemmas about Krull dimension which will be needed hereafter.

\smallskip

\begin{defn}
Let $M$ be a $R$-module. 
A prime ideal $\PP$ of $R$ is an \emph{associated prime ideal of $M$} if it is the annihilator of some element in $M$.
We denote by Ass$_R(M)$ the set of associated primes of $M$.
\end{defn}

\smallskip

\begin{rem}
Equivalently, a prime $\PP$ is an associated prime ideal of $M$ if $R/\PP$ is a submodule of $M$. Note that all the associated prime ideals of $M$ contain the annihilator of $M$.
\end{rem}

\smallskip

\begin{lem}[\cite{eisenbud}]\label{lem-ann=p}
Let $R$ be a Noetherian ring. Let $M, N$ and $Q$ be finitely generated $R$-modules fitting together in a short exact sequence
\begin{equation*}
N \hookrightarrow M \overset{p}{\twoheadrightarrow} Q.
\end{equation*}
Assume $\krull(N) < d$ and that there exists an ideal $\PP$ such that, for every $x \in {Q-\{0\}}$, $\Ann(x) = \PP$, with $\krull(R/\PP) = d$.

Then $\PP$ is an associated prime ideal of $M$. 
In particular, $M$ has a submodule isomorphic to $R/\mathcal P$.
\end{lem}

\begin{proof}
Denote by $I$ be the annihilator of $N$: $\krull(R/I) < d$ implies $I \nsubseteq \PP$.  Let $x \in I$, $x\notin \PP$. The module $xM$ has nontrivial image in $Q$ and one can check that $N \cap xM = \{0\}$. Hence $xM \simeq p(xM) \subset Q$ and elements of $xM$ have annihilator $\PP$.
\end{proof}

\smallskip

The next structure result will be very useful in the sequel. This shows that Noetherian modules are built up from modules isomorphic to $R/\mathcal P$ where $\mathcal P$ is prime.

\smallskip

\begin{prop}[\cite{eisenbud}]\label{prop-dec M}
Let $R$ be a Noetherian ring, $M$ be a finitely generated $R$-module. Then, there exist $M_0, M_1, \dots, M_n $ submodules of $M$ such that 
\begin{equation}\label{eq-dec M}
M = M_n > M_{n-1} > \dots > M_1 > M_0 = 0
\end{equation}
and $M_{i+1}/M_i \simeq R/\PP_i$, where $\PP_i$ is a prime ideal of $R$.
\end{prop}

As a consequence, and it is probably well-known, the Krull dimension of a Noetherian module only depends on its associated prime ideals.

\begin{prop}\label{prop-Krull dim associated prime}
Let $R$ be a Noetherian ring, $M$ be a finitely generated $R$-module.
Then, the Krull dimension of $M$ is attained by $R/\PP$ for some associated prime $\PP$ of $M$.
\end{prop}

\begin{proof}
Proposition \ref{prop-dec M} gives a decomposition of $M$ as a tower of rings of the form $R/\PP_i$, and Proposition \ref{prop-kdim mod max} implies that 
\[ \krull(M)  = \max_i\{ \krull( R/\PP_i)\}.\]
Look at the minimal $i$ such that the Krull dimension of $M_{i+1}/M_i$ is the Krull dimension of $M$. 
We can apply Lemma \ref{lem-ann=p} to the exact sequence
\begin{equation*}
M_i \hookrightarrow M_{i+1} \twoheadrightarrow M_{i+1}/M_i.
\end{equation*}
\end{proof}

\bigskip

\subsection{Krull dimension of a group}
\subsubsection{Definition}
Tushev defined in \cite{Tus03} the Krull dimension of a group analogously to that of a module. For a group $G$, let $\mathcal N(G) $ be the poset of all normal subgroups of $G$.

\begin{defn}[\cite{Tus03}]
Let $G$ be a group. We say that $G$ \emph{admits a Krull dimension} whenever the poset $\mathcal N(G)$ admits a deviation. In this case, we set
\begin{equation*}
\krull (G) = \dev \mathcal N(G).
\end{equation*}
Otherwise, $G$ \emph{does not admit a Krull dimension}.
\end{defn}

\begin{exs}\label{ex-Kd gps}~
\begin{enumerate}
\item A finite group has Krull dimension $0$.
\item $\Z$ has Krull dimension $1$. Indeed, a decreasing sequence of subgroups of $\Z$ is a sequence $(c_n \Z)_n$, where $c_n$ belongs to $\N$ and $c_n \mid c_{n+1}$. The factors are finite, except when $c_n$ is nonzero and $c_{n+1}$ is zero. This can only happen once.
\end{enumerate}
\end{exs}

\begin{rem}\label{rem-kd ab}
If the group is abelian, its Krull dimension coincide with its Krull dimension as a $\Z$-module.

Morevover, if $K, H$ are subgroups of an abelian group, the deviation of the factor $[H, K]$ is exactly the dimension of the quotient group $K/H$.
\end{rem}

\begin{lem}\label{lem- Krull dim ab}
A finitely generated abelian group has Krull dimension zero if it is finite, or one if it is infinite.
\end{lem}

\begin{proof}
Such a group $G$ is isomorphic to $\Z^d \times F$, for some integer $d$ and some finite group $F$.

Let $G_0 < ... < G_m$ be a series of subgroups of $G$ in which there are $n$ infinite factors.
For each infinite factor $G_j/G_{j-1}$ choose an element $x_j$ in $G_j$ which has
infinite order modulo $G_{j-1}$. Then these $x_j$ taken together generate a free abelian group of rank $n$ and hence $n \le d$.

One can derive an alternative proof from Remark \ref{rem-kd ab} above. By $(\ref{eq-Kd2} )$, the Krull dimension of $G$ is the Krull dimension of $\Z / \Ann(G)$, when $\Ann(G)$ denotes the annihilator of the $\Z$-module $G$. Therefore $\krull(G)$ is $1$ if the group is infinite, or $0$ if it is finite.
\end{proof}

A reformulation of Proposition \ref{prop-existence of Kd} yields: 

\begin{prop}
If $G$ satisfies the maximal condition on normal subgroups, then $G$ has a Krull dimension.
\end{prop}

This is not a necessary condition: the group $\Z[\frac{1}{p}]$ does admit a Krull dimension, equal to $1$.

\medskip

\subsubsection{Krull dimension of a $G$-group}~

To study the Krull dimension of a metabelian group and link it with the Krull dimension of certain submodules of the group, we need the following broader notion. Let $G, H$ be two groups. $H$ is said to be a \emph{$G$-group} if there is an action of $G$ on $H$ containing the inner automorphisms of $H$. If $H$ satisfies an exact sequence
$ \exseq{M}{H}{Q} $,
with abelian groups $M$ and $Q$, then the induced actions on $M$ and $Q$ endow them with a structure of $G$-groups.

\begin{defn}
Let $H$ be a $G$-group for some group $G$. Denote by  $\mathcal N_G(H)$ the subposet of subgroups of $H$ that are stable under the action of $G$. 
We say that $H$ \emph{admits a Krull dimension as a $G$-group} whenever the poset  $\mathcal N_G(H)$ admits a deviation. In this case, we set
\begin{equation*}
\krull_G (H) = \dev \mathcal N_G(H),
\end{equation*}
Otherwise, $H$ \emph{does not admit a Krull dimension as a $G$-group}.
\end{defn}

\smallskip

\begin{rems}~
\begin{enumerate}
\item The poset $\mathcal{N}_G(H)$ is a subposet of $\mathcal{N}(H)$, therefore $\krull_G(H) \leq \krull(H)$.
\item The Krull dimension of $G$ as a $G$-group for the conjugation action is indeed the Krull dimension of $G$. 

\item If $K, L$ are elements of $\mathcal{N}_G (H)$, with $K \subset L$, then 
\begin{equation*}
\dev [K, L] = \krull_G (L/K),
\end{equation*}
where $[K, L]$ is the the segment between $K$ and $L$ in $\mathcal{N}_G (H)$ on the left-hand side, and $L/K$ denotes the $G$-group for the induced action of $G$ on the right-hand side.

As a consequence, if $N$ is a subgroup of a $G$-group $H$, stable under the action of $G$, then $\krull_G (N) \leq \krull_G (H)$.

Moreover, if $Q$ is a quotient of a $G$-group $H$ by a $G$-stable subgroup, then the action of $G$ on $H$ induces a structure of $G$-group on $Q$ and again $\krull_G (H) \geq \krull_G (Q)$.
\end{enumerate}
\end{rems}

\begin{lem}\label{lem-kdim g gps}
Let 
\begin{equation*}
M \hookrightarrow H \twoheadrightarrow Q
\end{equation*}
be a sequence of $G$-groups.
Then, 
\begin{equation*}
\krull_G (H) = \max \{ \krull_G (M), \krull_G (Q)\}.
\end{equation*}
\end{lem}

\begin{proof}
The fact that $\krull_G (H) \geq \max\{\krull_G (M), \krull_G (Q) \}$ follows from the remarks.

To show the reverse inequality, we use induction on the ordinal 
\begin{equation*}
q = \max \{ \krull_G (M), \krull_G (Q)\}.
\end{equation*}
It holds if $q$ is $- \infty$ or $0$. Assume that the reverse inequality is true for any such extension where the corresponding maximum is striclty less than $q$.

Let $(M_n)_n$ be a decreasing sequence of $\mathcal{N}_G (H)$. We need to show that all but finitely many of the factors $[M_{n+1}, M_{n}]$ have deviation less than $q$. Equivalently, we need to show that all but finitely many of the $G$-groups $M_n / M_{n + 1}$ have dimension less than $q$. Denote by $p$ the projection of $H$ into $Q$ and set $I_n = M_n \cap H$ and $P_n = p(M_n)$: these are decreasing sequences of $\mathcal{N}_G(M)$ and $\mathcal{N}_G(Q)$ respectively.
By definition of $q$, all but finitely many of the factors of these sequences have deviation less than $q$.

To apply the induction hypothesis, consider $S_n = M_{n+1}I_n$ in $\mathcal{N}_G(H)$. We have 
\begin{equation*}
S_n / M_{n +1} \hookrightarrow M_n / M_{n + 1} \overset{p}{\twoheadrightarrow} M_n / S_n
\end{equation*}
and $M_n / S_n \simeq P_n / P_{n+1}$ and $S_n /M_{n+1} \simeq I_n / I_{n+1}$, hence they have dimension less than $q$: 
\begin{equation*}
\krull_G (M_n / M_{n +1}) \leq \max \left\{\krull_G (I_n / I_{n +1}), \krull_G (P_n / P_{n + 1})\right\} < q.
\end{equation*}
\end{proof}

\subsubsection{Krull dimension of a metabelian group}~

The derived series $(G^{(j)})_j$ of  a group $G$ is defined inductively by $G^{(0)} = G$ and $G^{(j+1)} = [G^{(j)}, G^{(j)}]$. The group $G$ is \emph{solvable} if $G^{(j)} = \{e\}$ for some $j$, and the smallest such $j$ is the \emph{derived length of $G$}.
Recall that a metabelian group $G$ is a solvable group with derived length $2$.
We will denote by $G_{ab}$ the abelianization $G/[G, G]$ of $G$.
The group $G$ fits into the exact sequence
\begin{equation*}
[G, G] \hookrightarrow G \twoheadrightarrow G_{ab}.
\end{equation*} 

\medskip

We now consider a general extension of one abelian group by another. Let $G$ be a metabelian group such that
\begin{equation*}
M \hookrightarrow G \twoheadrightarrow Q ,
\end{equation*} 
where $M$ and $Q$ are abelian groups. The group $G$ acts on $M$ by conjugation, and the action of an element $g$ in $G$ only depends on its projection in the quotient group $Q$. Hence, it induces an action from $Q$ on $M$, endowing $M$ with a structure of $\Z Q$-module. If $Q$ is finitely generated then $\mathbb ZQ$ is Noetherian and if $G$ is finitely generated
then both $\mathbb Z Q$ and the $\Z Q$-module $M$ are Noetherian. 

If $G$ is a finitely generated metabelian group, note that $Q$ is virtually $\Z^d$, for some $d$. Hence, the group $G$ has a finite index metabelian subgroup $G_1$ with torsion-free quotient $\Z^d$: 
\begin{equation}\label{eq-meta gp Zd}
 M \hookrightarrow G_1 \twoheadrightarrow \Z^d,
\end{equation} 
and the ring $\Z \Z^d$ identifies with $\Z[X_1^{\pm 1}, \dots, X_d^{\pm 1}]$.

\begin{prop}\label{prop-Kd meta}
Let $G$ be a metabelian group, admitting an exact sequence of the form
\begin{equation*}
M \hookrightarrow G \twoheadrightarrow Q,
\end{equation*}
where $M$ and $Q$ are abelian groups.
The Krull dimension of $M$ as a $G$-group for the conjugation action coincide with the Krull dimension of $M$ as a $\Z Q$-module. 

Then, if $\krull (M) > 0$, we have $\krull (G) = \krull (M)$. Otherwise, $\krull (G)$ is $0$ if the group is finite, $1$ if not.
\end{prop}

\begin{proof}
This follows from Lemmas \ref{lem- Krull dim ab} and \ref{lem-kdim g gps} , applied to the action of $G$ on itself by conjugation. Note that $\krull_G (Q)$ is the Krull dimension of the group $Q$.
\end{proof}

As a consequence, the Krull dimension of a metabelian group $G$ is the Krull dimension of $[G, G]$ as a $\Z G_{ab}$-module, except when the former is zero and the group infinite: we retain especially
\begin{equation*}
\krull (G) = \krull_{\Z G_{ab}}([G, G]), \text{ if positive}.
\end{equation*}

In particular, if the dimension is at least $2$, we may use any exact sequence expressing $G$ as an extension of an abelian group by another one to compute the Krull dimension of $G$ as the Krull dimension of the module involved. 

\begin{rem}
An easy consequence of this proposition is that finitely generated metabelian groups have finite Krull dimension.

In general, a metabelian group $G$ admits a Krull dimension if and only if it has a finite series of normal subgroups each of whose factor meets the maximal or the minimal condition for $G$-invariant subgroups. Recall that a $G$-group $H$ is said to satisfy the \emph{maximal (minimal) condition for $G$-invariant subgroups} if every descending (ascending) chain of $G$-invariant subgroups eventually terminates.
\end{rem}

Up to passing to a finite index subgroup, we may consider exact sequences with torsion-free abelian quotient as \ref{eq-meta gp Zd}.

\begin{prop}\label{prop-fin ind}
Let $G$ be a finitely generated metabelian group. Consider an exact sequence
\begin{equation*}
M \hookrightarrow G \twoheadrightarrow Q,
\end{equation*}
with $M$ and $Q$ abelian.

There exists a subgroup $G'$ of finite index in $G$ such that
\begin{enumerate}
\item $\krull (G) = \krull (G')$.
\item $G'$ is an extension of an abelian group by a finitely generated free abelian group.
\end{enumerate}
\end{prop}

\begin{proof}
Let $p: G \rightarrow Q$ be the projection. $Q$ is a finitely generated abelian group: we may write $Q = Q' \times T$ where $T$ is a finite abelian group and $Q'$ is finitely generated free abelian. Take $G' = p^{-1}(Q')$, it has finite index in $G$ and 
\begin{equation*}
M \hookrightarrow G' \twoheadrightarrow Q'.
\end{equation*}
Hence, we are left to show that $\krull (G) = \krull (G')$. We show $\krull_{\Z Q} (M) = \krull_{\Z Q'} (M)$.
Take a decomposition of $M$ as a $\Z Q$-module, as given in \ref{eq-dec M}: 
$M = M_n \supset M_{n - 1} \supset \dots \supset M_1 \supset \{0\}$, where $M_{i + 1}/M_i = \Z Q/ \PP_i$, for some prime $\PP_i$. Then,
\begin{equation*}
\krull_{\Z Q} (M) = \max_i \{\krull (\Z Q/ \PP_i)\}. 
\end{equation*}
The following lemma implies that
\begin{equation*}
\krull (\Z Q/ \PP_i) = \krull_{\Z Q'/(\PP_i \cap Q')} (\Z Q/ \PP_i) = \krull_{\Z Q'} (\Z Q/ \PP_i)
\end{equation*}
and $\max_i \{\krull_{\Z Q'} (\Z Q/ \PP_i)\} = \krull_{\Z Q'} (M)$.
When this dimension in nonzero, Proposition \ref{prop-Kd meta} ensures that we are done. When it is zero, $\krull (G)$ is 0 if the group is finite, $1$ if not. As $G'$ has finite index, we have the same dichotomy.
\end{proof}

\begin{lem}
Let $A \supset B$ be two rings such that $A$ is finitely generated as a $B$-module. Then,
\begin{equation*}
\krull (A) = \krull (B) = \krull_B (A).
\end{equation*}
\end{lem}

\begin{proof}
Proposition $9.2$ of \cite{eisenbud} states that $\krull (A) = \krull (B)$. We compare this quantity with $\krull_B (A)$. First,
$\krull_B (A) = \dev \mathcal L_B(A) \geq \dev \mathcal L_A(A) = \krull (A)$.
On the other hand, 
$\krull_B (A) \leq \krull (B) = \krull (A).$
\end{proof}

Similarly to the case of modules, we have the following refinements. Proposition \ref{prop-fin ind} ensures that this is well-defined.

\begin{defn}
Let $G$ be a metabelian group, satisfying the exact sequence
\begin{equation*}
\exseq{M}{G}{Q}
\end{equation*}
for some abelian groups $M$ and $Q$ so that $\krull_{\Z Q}(M) >0$.
The \emph{torsion-free Krull dimension} of $G$ is $\krull^0(G) = \krull^0(M)$ and
the \emph{torsion Krull dimension} of $G$ is $\krull^t(G) = \krull^t(M)$.
\end{defn}

Again, if $\krull_{\Z Q}(M)$ is nonzero, we have
\begin{equation*}
\krull(G) = \max\{\krull^0(G), \krull^t(G)\}.
\end{equation*}

\subsection{Examples}

We study now three classes of examples. First, we consider small dimensional metabelian groups and study the rank of the torsion-free ones. Then, in the second and third paragraphs, we give the Krull dimension of some metabelian wreath products and of the free ($p$-)metabelian groups (see the definitions below). These two last classes will appear to be fundamental in the sequel.

\medskip

\subsubsection{Small dimensional metabelian groups}~

Finitely generated metabelian groups of Krull dimension $0$ are finite, as stated in Proposition \ref{prop-Kd meta}. When the dimension is $1$ and the group is torsion-free, we can also say something about the structure. Recall that a group has \emph{finite Prüfer rank} if there exists an integer $r$ such that any finitely generated subgroup can be generated by at most $r$ elements. The least such $r$ is the \emph{Prüfer rank} of the group. 

\begin{prop}\label{prop-dim 1 torsion-free}
Let $G$ be a finitely generated torsion-free metabelian group of dimension one. Then, it has finite Prüfer rank.
\end{prop}

This Proposition is false if we do not assume that the group is torsion-free: we shall see later in Section \ref{sec-specsub} that a finitely generated metabelian group of dimension $1$ whose derived subgroup is torsion has a subgroup isomorphic to a lamplighter.

\begin{proof}
Let $G$ be a finitely generated torsion-free metabelian group of dimension $1$. By Proposition \ref{prop-fin ind}, up to passing to a finite index subgroup, we may assume that $G$ fits inside an exact sequence 
\begin{equation*}\label{eq-exseq dim 0}
\exseq{M}{G}{\Z^d},
\end{equation*}
for some integer $d$, and abelian group $M$.
The group $M$ is a Noetherian $\Z\Z^d$-module and, by Proposition \ref{prop-dec M}, it admits an increasing sequence of submodules $M_i$ whose factors have the form $\Z\Z^d /\PP$, for $\PP$ a prime ideal.

We first consider the case where $M$ is a ring. It has characteristic zero, and Proposition \ref{prop-Kd trdeg} implies that the transcendental degree of its fraction field over $\Q$ is zero, so $M$ is an algebraic number field.
Algebraic number fields have finite Prüfer rank. Hence, $M$ has finite Prüfer rank.  

In the general case, we look at a decomposition of $M$: the subquotients $M_{i+1}/M_i$ have dimension either zero or one. If the dimension is one, it is enough to show that the characteristic is not positive. By contradiction, suppose $M_{k+1}/M_k$ has positive characteristic $p$ and dimension $1$. $M_{k+1}/M_k \simeq \Z \Z^d/ \PP$, with $\PP$ prime, hence $p$ is prime . Then, the transcendental degree of the fraction field of $M_{k+1}/M_k$  is $1$, and Proposition \ref{prop-tr e} will imply that it contains $\mathbb F_p[X^{\pm1}]$. Pulling back the transcendental element $X$ in $M_{k+1}$: we still get a transcendental element, which contradicts the fact that $M$ is torsion-free of dimension $1$.

The property of being of finite Prüfer rank is stable under extension, and finitely generated abelian groups have finite Prüfer rank. Hence, as $M$ has finite Prüfer rank, $G$ has finite Prüfer rank as well. 
\end{proof}

\medskip

\subsubsection{Wreath products}~

Let $K$ and $H$ be countable groups. The wreath product $K \wr H$ with base group $H$ is the semi-direct product $(\oplus_{h \in H} K_h) \rtimes H$ where $K_h$ are copies of $K$ indexed by $H$, and $H$ acts by translation of the indices. More precisely, one may identify elements in $\oplus_{h \in H} K_h$ with finitely supported functions from $H$ to $K$. Thus, for any two elements $(f, h), (f, h')$ in the group, their multiplication is given by
\begin{equation*}
(f, h)(f', h') = (f f'(h^{-1} \bullet), hh').
\end{equation*}

Similarly, the unrestricted wreath product of $K$ by $H$, denoted by $K \wr \wr H$ is the semi-direct product $(\sum_H K_h) \rtimes H$.

The most classical example of a wreath product is the lamplighter group $(\Z/2\Z) \wr \Z$. This group is generated by two elements $(0, 1)$ and $(\delta_0, 0)$, where $\delta_0$ is the function from $\Z$ to $\Z/2\Z$ taking value $1$ in $0$, and vanishing elsewhere. This group can be described in the following way: imagine a bi-infinite pathway of lamps, one for each integer. Each one of those lamps can be on or off, and a lamplighter is walking along the line. An element of the group corresponds to a configuration for the lamps, with only finitely many of them lit, and a position for the lamplighter. The first generator $(0, 1)$ changes the position of the lamplighter while the second generator $(0, \delta_0)$ has him change the state of the lamp at his current position.
A configuration is a function from $\Z$ to $\Z/2\Z$ and  corresponds to an element in $\mathbb F_2[X, X^{-1}]$. This subgroup has dimension $1$ as a $\Z[X, X^{-1}]$-module. 

In other words, the lamplighter group $(\Z/2\Z) \wr \Z$  has an elementary abelian subgroup of dimension $1$ such that the quotient is $\Z$. Hence, its Krull dimension is $1$.

Other examples are :
\begin{enumerate}
\item $\krull(\Z \wr \Z^d) = d + 1$.
\item $\krull( F \wr \Z^d) = d$, where $F$ is a finite group.
\end{enumerate}

\medskip

\subsubsection{Free ($p$-)metabelian groups}~

Denote by $F_d$ the free group of rank $d$. Let $F_d'$ be its derived subgroup and $F_d''$ the second term of its derived series.

The \emph{free metabelian group of rank $d$} is 
$$B_d^{(p)} := F_d / F_d''$$ and the \emph{free $p$-metabelian group of rank $d$} is 
$$B_d^{(p)} := F_d / F_d''(F_d')^p.$$

$B_d^{(p)}$ is the freest metabelian group whose derived subgroup has exponent $p$ and is the quotient of the free metabelian group of rank $d$, $B_d$ by the $p$-powers of its commutators.

These two groups have abelianization $\Z^d$ and looking at the associated exact sequence, one may compute their dimensions:
\begin{equation*}
\krull(B_d) = d + 1, \text{ and }\krull(B_d^{(p)}) = d.
\end{equation*}

\section{Preliminaries}\label{sec-prelim}

\subsection{Generalities on the return probability}\label{ssec-return proba}

We review some known properties and known behaviours for $p_{2n}$ to provide a larger picture to the reader.

\subsubsection{Stability properties } (\cite{PSC00})
\begin{enumerate}
\item Let $H$ be a finitely generated subgroup of $G$. Then $p_{2n}^H \succsim p_{2n}^G$.
\item Let $Q$ be a quotient of $G$. Then $p_{2n}^Q \succsim p_{2n}^G$.
\end{enumerate}
If the index of the subgroup, respectively the kernel of the quotient, is finite, then the inequality is an  equivalence.

As a consequence, when studying the return probability of a finitely generated metabelian group of a given Krull dimension, we may assume that its admits an exact sequence such as \ref{eq-meta gp Zd}, up to passing to a finite index subgroup (see Proposition \ref{prop-fin ind}).

\subsubsection{Known behaviours: amenability and growth} The following results were mentionnend in the introduction.

\begin{enumerate}
\item $G$ has polynomial growth of degree $d$ iff $p_{2n}^G \sim n^\frac{-d}{2}$ (see \cite{Var} and \cite{HSC}).
\item $G$ is non-amenable iff $p_{2n}^G \sim \exp(-n)$ (\cite{Kesten59}).
\item If $G$ has exponential growth then $p_{2n}^G \precsim \exp(-n^\frac{1}{3})$ (\cite{HSC}).
Moreover, if $G$ is a discrete subgroup of a connected Lie group, there is an equivalence: $G$ is amenable with exponential growth iff $p_{2n}^G \sim \exp(-n^\frac{1}{3})$.

These three previous behaviours are the only possible ones in this case.
\end{enumerate}

\smallskip

In the introduction, we mentioned that outside the world of discrete subgroups of connected Lie groups, far more behaviours happen. Here are some other examples.

\subsubsection{Known behaviours: among solvable groups}
\begin{enumerate}
\item Lamplighter groups:
\begin{itemize}
\item  $\Z \wr \Z^d :p_{2n}^\Z \wr \Z^d \sim \exp(-n^\frac{d}{d+2}(\log n)^\frac{2}{d+2})$,
\item  $F \wr \Z^d: p_{2n}^F \wr \Z^d \sim \exp(-n^\frac{d}{d+2})$ (see \cite{Erschler06}).
\end{itemize}
\item Free solvable groups on $d$ generators :
\begin{itemize}
\item Free metabelian group $B_d$: $p_{2n}^{B_d} \sim \exp(-n^\frac{d}{d+2}(\log n)^\frac{2}{d+2})$,
\item Free $r$-solvable group($r> 2$) $B_{d, r}$: $p_{2n}^{B_{d,r}} \sim \exp(-n (\frac{\log_{[r-1]} n}{\log_{[r-2]}n})^\frac{2}{d})$ (\cite{SCZ14}). 
\end{itemize}
\end{enumerate}

\smallskip

\subsection{Return probability of $B_d^{(p)}$}~

In \cite{SCZ14}, Saloff-Coste and Zheng computed the asymptotic of the return probability of the free solvable group of rank $d$ (and, in particular, that of the free metabelian group of rank $d$), using a method based on the Magnus embedding.
In this part, we explain why it is possible to use their techniques to get the return probability of the free $p$-metabelian group, obtained by adding $p$-torsion to the derived subgroup. The only thing to do is to check that the Magnus embedding behaves well in this case: this was done by Bachmuth in \cite{Bach67}.

\bigskip

Recall that the Magnus embedding allows to embed groups of the form $F_d/[N, N]$, with $N$ a normal subgroup of $F_d$, into the wreath product $\Z^d \wr (F_d/N)$ (\cite{Magnus39}). 

In the case of $B_d$, the free metabelian group of rank $d$, this applies to $N = F_d'$. The resulting wreath product is $\Z^d \wr \Z^d$, that is,the semi-direct product $\oplus_{\Z^d} \Z^d \rtimes \Z^d$, where $\Z^d$ acts by shift on $\oplus_{\Z^d} \Z^d$.

The embedding takes the following form. 
Let $s_1, \dots, s_d$ denote the classical generators of $F_d$, and $a_1, \dots, a_d$ their images in the abelianization $\Z^d$. 

Consider matrices of the form
\begin{equation*}
\begin{pmatrix}
   b & m \\
   0 & 1 
\end{pmatrix}
\end{equation*}
with $b$ in $\Z^d$ and $m \in M$, the free $\Z(\Z^d)$-module of rank $d$ with basis $(e_i)_{i = 1, \dots, d}$. This is a matricial representation of $\Z^d \wr \Z^d$. Let $i$ be the extension to a homomorphism of 
\begin{equation*}
s_i \mapsto \begin{pmatrix}
   a_i & e_i \\
   0 & 1 
\end{pmatrix}.
\end{equation*}

Magnus proved that the kernel of this homomorphism is $[[F_d, F_d], [F_d, F_d]]$ (\cite{Magnus39}). Therefore, $i$ induces an embedding from $B_d$ to $\Z^d \wr \Z^d$.

Bachmuth studied this representation in the case of free $k$-metabelian groups, for any positive integer $k$. He proved that moding out by $k$ in the derived subgroup corresponds to moding out by $k$  in $M$, that is

\begin{prop}[\cite{Bach67}, Proposition 1]
The Magnus embedding induces an embedding of the free $k$-metabelian group of rank $d$, $B_d^{(k)}$ into $(\Z/k\Z)^d \wr \Z^d$. 
\end{prop}

A lower bound for $p_{2n}^{B_d^{(k)}}$ follows directly from the comparison with $(\Z/k\Z)^d \wr \Z^d :$  \[p_{2n}^{B_d^{(k)}} \succsim \exp(-n^\frac{d}{d+2}).\]

In \cite{SCZ14}, Saloff-Coste and Zheng used the Magnus embedding to produce upper bounds for the return probability of groups of the form $F_d/[N, N]$, for some normal subgroup $N$ of $F_d$. 

They introduced the notion of exclusive pair in such a group $F_d/[N, N]$, made of a subgroup $\Gamma$ together with an element $\rho$ of the derived subgroup. The pair is designed so that the images of  $\Gamma$ and $\rho$ in the $\Z$-module $M$ should have minimal interaction (we refer to \cite{SCZ14}, $\mathsection 4$ for a precise definition). From this, they derived a comparison (\cite{SCZ14}, Theorem $4.13$): the return probability of $F_d/ [N, N]$ is smaller than the return probability of the subgroup $\Z \wr \bar{\Gamma}$ in $\Z^d \wr (F_d/N)$, where $\bar{\Gamma}$ is the image of $\Gamma$ in $F_d/N$.

In the case of the free metabelian group $B_d$, $\Gamma$ can be chosen so that $\bar{\Gamma}$ has finite index in $\Z^d$ (more generally, this is possible whenever $F_d/N$ is nilpotent).

Their techniques still apply to the free $k$-metabelian groups, if one consider $\Z/k\Z$-modules instead of $\Z$-modules and yield
\[p_{2n}^{B_d^{(k)}} \precsim p_{2n}^{\Z/k\Z \wr \Z^d} \precsim \exp(-n^\frac{d}{d+2}). \]

As a conclusion, we get

\begin{prop}\label{prop-return proba B_2^p}
The return probability of the free $k$-metabelian group of rank $d$ is equivalent to $\exp(-n^\frac{d}{d +2})$.
\end{prop}

\subsection{An embedding theorem for metabelian groups that preserves the Krull dimension}\label{ssec-embedding}

In this section, we consider a finitely generated metabelian group $G$, extension of an abelian group $M$ by a finitely generated abelian group $Q$. Recall that $M$ also carries a structure of $\Z Q$-module.

A famous theorem of Kaloujinine and Krasner allows to embed $G$ in a split metabelian group. 

\begin{thm}[see \cite{KM}]
Any extension of a group $A$ by a group $B$ embeds in the non-restricted wreath product $A \wr \wr B$ of $A$ with $B$.
\end{thm}

We first recall their construction and then explain how to modify it so as to preserve the Krull dimension. 

\medskip

\subsubsection{The embedding of Kaloujinine and Krasner}
~

Fix a section $s : Q \rightarrow G$ of the group. The embedding $i: G \hookrightarrow M \wr\wr Q = \sum_Q M \rtimes Q$ is given by
\begin{equation*}
g \mapsto (f_g, \bar{g}).
\end{equation*}
where $\bar{g}$ is the image of $g$ in the quotient $Q$ and 
\begin{equation*}
f_g: \begin{cases}
Q \rightarrow M \\
q \mapsto s(gq)^{-1}gs(q).
\end{cases}
\end{equation*}
Note that the embedding does depend on $s$. This embedding commutes with the projection on $Q$, and its restriction to $M$ is $\Z Q$-equivariant. 
Let $S$ be a finite generating set for $G$ and $\pi : M \wr \wr Q \rightarrow \sum_Q M$ denote the canonical projection. Set $T = \pi(i(S))$, it is a finite subset of the base of the wreath product $M\wr\wr Q$. The subgroup $\hat G$, generated by $T$ and $Q$ in $M\wr\wr Q$, contains $i(G)$ and is the semi-direct product $B \rtimes Q$, where $B$ is the submodule of $\sum_Q M$ generated by $T$.

Therefore, $G$ embeds in the finitely generated split metabelian group $\hat G$.

\medskip

\subsubsection{Respecting the Krull dimension}~

We now explain how to modify this embedding so that the Krull dimension of the target group equals $\krull(G)$. 

Let \begin{equation*}
\mathcal{S} = \left\{ C \subset B \text{ submodule } \mid C\cap i(G) = \{0\} \right\}.
\end{equation*}

Zorn's lemma provides us a maximal element $C_0$ of $\mathcal S$. Set $B_0 = B/C_0$. By construction, $G$ embeds in $B_0 \rtimes Q$. We still denote by $i$ the composition of $i$ with the projection onto $B_0 \rtimes Q$. Note that every submodule of $B_0$ now intersects $i(G)$.

\textbf{Claim.} The associated prime ideals of $B_0$ and $M$ are the same. 

Indeed, Ass$(M) \subset$ Ass$(B_0)$, because  $M\subset B$. On the other hand, if $\PP$ is an associated prime of $B_0$, then $\Z Q/\PP$ is a submodule of $B_0$, hence does intersect $i(G)$: there exists $g$ in $G$ such that $i(g) \in \Z Q/\PP$. More precisely, $g$ projects trivially onto $Q$ and thus belongs to $M$.  It generates a submodule isomorphic to $\Z Q/\PP$ in $M$. 

Lemma \ref{prop-Krull dim associated prime} then implies that $M$ and $B_0$ have the same Krull dimension. Now (see Proposition \ref{prop-Kd meta}), if this dimension is positive, this is also the dimension of $G$ and $B_0 \rtimes Q$, and we are done. Otherwise, $M$ and $B_0$ have dimension $0$. If $G$ is infinite, so is $B_0 \rtimes Q$, and they have dimension $1$. The last case arises when $\krull (G) = \krull (M) = 0$, that is $G$ is finite: $B_0 \rtimes Q$ is finite as well, hence has dimension $0$. 

\smallskip

We just proved

\begin{thm}\label{thm-embedding}
Let $G$ be a metabelian group, given as an extension
\begin{equation*}
\exseq{M}{G}{Q},
\end{equation*}
with $M$ and $Q$ abelian.

Then, there exists an embedding of $G$ inside a split metabelian group $B \rtimes Q$, commuting with the projection on $Q$, such that $M$ and $B$ have the same associated prime ideals. In particular, $\krull(G) = \krull(B \rtimes Q)$.

Moreover, if $G$ is finitely generated, $B \rtimes Q$ can be chosen to be finitely generated as well.
\end{thm}

\section{Special subgroups of metabelian groups}\label{sec-specsub}

The purpose of this section is to prove 

\begin{prop}\label{prop-special subgroups}
Let $G$ be a metabelian group. Assume $G$ has Krull dimension at least $2$. Then, $G$ has a subgroup isomorphic to either $\Z \wr \Z$, or to $B_2^{(p)}$ for some prime $p$. 

The first option happens whenever $\krull^0(G) \geq 2$, and the consequence for the return probability of $G$ is 
\begin{equation*}
p_{2n}^G \precsim \exp\left(- n^\frac{1}{3}(\log n)^\frac{2}{3}\right).
\end{equation*} 

The second option happens whenever $\krull^t(G) \geq 2$, and yields
\begin{equation*}
p_{2n}^G \precsim \exp\left(- n^\frac{1}{2}\right).
\end{equation*} 
\end{prop}

\subsection{Looking for transcendental elements}\label{ssec-trans}

We first show that a group of dimension $k$ contains a maximal polynomial ring of the same dimension.

\begin{prop}\label{prop-max family trans e}
Let $G$ be a finitely generated metabelian group, satisfying
\begin{equation}\label{eq-exseq trans mod}
\exseq{M}{G}{Q},
\end{equation}
with $M, Q$ abelian. Assume that $G$ has Krull dimension $k \geq 1$.

Then, $M$ contains a ring isomorphic to  $\Z[X_1^{\pm 1}, \dots, X_{k-1}^{\pm 1}]$ or $\mathbb{F}_p [X_1^{\pm 1}, \dots, X_k^{\pm 1}]$, for some prime $p$. The first option happens whenever $\krull^0(G) = k$, the second whenever $\krull^t(G) = k$.
\end{prop}

It will be enough to deal with the case of $M$ being a ring, thanks to Propositions \ref{prop-dec M} and \ref{prop-Krull dim associated prime} on the structure of $M$.

\medskip

\subsubsection{The case of a ring}~

In this paragraph, let $Q$ be a finitely generated free abelian group and $A$ be a ring isomorphic to $\Z Q / \PP$, for some prime ideal $\PP$. The group $Q$ acts on $A$ by multiplication by the corresponding monomial. The characteristic of $A$ is either a prime $p$, or zero.
Denote by $K$ be the fraction field of $A$, and by $\mathbb P$ its prime field.

\begin{prop}\label{prop-tr e}
Let $D$ be the degree of transcendence of $K$ over $\mathbb P$. 

Then $A$ contains a family of $D$ transcendental monomials which is algebraically free.
The subring generated by this family is, according to the characteristic, either of the form 
\[\Z[X_1^{\pm 1}, \dots , X_D^{\pm 1}] \text{ ~~or~~ } \Z/p\Z [X_1^{\pm 1}, \dots , X_D^{\pm 1}], \] 
for some prime $p$. 

As a consequence, if $A$ has Krull dimension $d$, we have the following dichotomy :
\begin{itemize}
\item when $A$ has characteristic $p$, $A$ has a subring isomorphic to  $\Z/p\Z [X_1^{\pm 1}, \dots , X_d^{\pm 1}]$.
\item when $A$ has characteristic zero, $A$ has a subring isomorphic to $\Z[X_1^{\pm 1}, \dots , X_{d-1}^{\pm 1}]$.
\end{itemize}

\end{prop}

\begin{lem}
Let $C$ be a subfield of $K$. If the degree of transcendence of $K$ over $C$ is positive then there is a transcendental monomial in $A$.
\end{lem}

\begin{proof}
By hypothesis, there exists in $K$ an element which is transcendent over $C$. We can write it $\frac{t_1}{t_2}$, with $t_1, t_2 \in A$. The set of all elements of $K$ that are algebraic over $C$ is a subfield of $K$, therefore one of the $t_i$'s has to be transcendental. 
This transcendental element belongs to $A$ and is a polynomial modulo $\PP$, hence a linear combination of monomials. Again, one of its monomials has to be transcendental.
\end{proof}

We may now proceed to the proof of the proposition.

\begin{proof}[Proof of Proposition \ref{prop-tr e}]
The proof is by induction on $D$.

Applying the previous lemma with $C = \mathbb P$ yields a monomial $m_1 \in A$, transcendental over $\mathbb P$.
Let $k \leq D$. Suppose we have constructed an algebraically free family $(m_1, \dots, m_{k-1})$ of transcendental monomials of $A$ over $\mathbb P$. Then, the transcendence degree of $K$ over $\mathbb P(m_1, \dots , m_{k-1})$ is $D - k + 1 > 0$. We apply the lemma again with $C = \mathbb P(m_1, \dots , m_{k-1})$ to get $m_k$.
\end{proof}

\subsubsection{Proof of proposition \ref{prop-max family trans e}}

\begin{proof}

Applying Proposition \ref{prop-fin ind}, we may assume that $G_{ab}$ is free abelian of rank $d$. Hence, $G$ fits into an exact sequence such as (\ref{eq-exseq trans mod}) with quotient $Q \simeq \Z^d$, for some $d$.
As a consequence of Lemma \ref{lem-dec of M}, one can write the module $M$ as $M = T(M) \oplus M_0$, where $T(M)$ is the torsion subgroup of $M$ and $M_0$ is torsion-free. We have
\begin{equation*}
\krull^0(G) = \krull(M/T(M)) \text{ and } \krull^t(G) = \krull(T(M)).
\end{equation*}
Proposition \ref{prop-Krull dim associated prime} implies that $T(M)$, respectively $M/T(M)$, contains a ring isomorphic to $\Z \Z^d/ \PP$ of Krull dimension $\krull(T(M))$, respectively $\krull(M/T(M))$. 

Conclusion then follows from Propositions \ref{prop-Kd trdeg} and \ref{prop-tr e}, up to pulling back the polynomial ring in the second case.
\end{proof}

\subsection{Wreath products inside metabelian groups} 

We use the results of the previous subsection to exhibit wreath products as subgroups of some (split) metabelian groups.

\subsubsection{In split metabelian groups}

\begin{prop}\label{prop-split case}
Let $G$ be a split finitely generated metabelian group, that is ${G = M \rtimes Q}$, with $M$ and $Q$ abelian. Let $k$ be the Krull dimension of $G$ and assume $k \geq 1$. Then, 
\begin{enumerate}
\item if $\krull^0(G) = k$, then $G$ has a subgroup isomorphic to $\Z \wr \Z^{k-1}$. As a consequence,
\begin{equation*}
p_{2n}^G \precsim p_{2n}^{\Z \wr \Z^{k-1}} \sim \exp\left(-n^\frac{k-1}{k+1} (\log n)^\frac{2}{k+1}\right).
\end{equation*}
\item otherwise, there exists a prime $p$ such that $G$ has a subgroup isomorphic to $(\Z/ p\Z) \wr \Z^k$. As a consequence,
\begin{equation*}
p_{2n}^G \precsim p_{2n}^{(\Z/p\Z) \wr \Z^k} \sim \exp\left(-n^\frac{k}{k+2}\right).
\end{equation*}
\end{enumerate}
\end{prop}

\begin{proof}
Up to passing to a finite index subgroup, we may assume that $Q$ is a finitely generated free abelian group (Propositon \ref{prop-fin ind}). This does not change the dimension nor the return probability.
In both cases, Proposition \ref{prop-max family trans e} implies that $M$ contains a ring isomorphic to either $\Z[X_1^{\pm 1}, \dots X_{k-1}^{\pm 1}]$, or to $\mathbb{F}_p [X_1^{\pm 1}, \dots X_k^{\pm 1}]$ for some prime $p$. We denote this ring by $B$.

Recall that these algebraically free transcendental elements are actually monomials of $Q \simeq \Z^d$. The subgroup of $G$ generated by $B$ and these monomials is either $\Z \wr\Z^{k-1}$ or $(\Z/p\Z) \wr \Z^k$.
\end{proof}

\medskip

\subsubsection{In general metabelian groups}~

In the general case, we may not be able to lift commutatively the transcendental elements obtained as we do in Proposition \ref{prop-split case}. 

For instance, in the free metabelian group $B_2$, the images of the two generators of $\Z^2$ by any section $s$ from $\Z^2$ to $B_2$ will never commute.

\bigskip

We still have the following corollary.

\begin{cor}\label{cor-ZZ}
Let $G$ be a metabelian group. If $\krull^0(G) \geq 2$, then $G$ has a subgroup isomorphic to $\Z\wr\Z$ and 
\begin{equation*}
p_{2n}^G \precsim p_{2n}^{\Z \wr \Z} \sim \exp\left(-n^\frac{1}{3} (\log n)^\frac{2}{3}\right).
\end{equation*}
\end{cor}

This proves the first part of Proposition \ref{prop-special subgroups}. The next part deals with the case $\krull^t(G) \geq 2$.

\subsection{Extensions of torsion modules by $\Z^2$}

Let $p$ be a prime. The $p$-metabelian free group of rank $2$, $B_2^{(p)}$ fits into an exact sequence
\begin{equation}\label{eq-exseq B_2^p}
\exseq{[B_2^{(p)}, B_2^{(p)}]}{B_2^{(p)}}{\Z^2}.
\end{equation}
Denote by $\alpha$ and $\beta$ its two generators. As a group, $[B_2^{(p)}, B_2^{(p)}]$ is generated by conjugates of $[\alpha, \beta]$. Hence, it is generated by $[\alpha, \beta]$ as a $\mathbb F_p[X^{\pm 1}, Y^{\pm 1}]$-module and $[B_2^{(p)}, B_2^{(p)}]$ is a cyclic module. As it is also a torsion-free module, it is isomorphic to $F_p[X^{\pm 1}, Y^{\pm 1}]$.

\begin{lem}\label{lem-non split extension of Z2}
Let $H$ be an extension of $\mathbb F_p[X^{\pm 1}, Y^{\pm 1}]$ by $\Z^2$, with the usual action of $\Z^2$ on $\mathbb F_p[X^{\pm 1}, Y^{\pm 1}]$. Then $H$ is metabelian and has a subgroup isomorphic to $B_2^{(p)}$.
\end{lem}

\begin{proof}
$H$ is metabelian because the action is not trivial.

Take $a, b$  in $H$ so that they do not commute and so that their projections in $\Z^2$ generate $\Z^2$.
We claim that $\langle a, b \rangle$ is isomorphic to $B_2^{(p)}$.
Indeed, $\langle a, b \rangle$ is a metabelian group generated by two elements. Its derived subgroup is the ideal $([a, b])$ in $\mathbb F_p[X^{\pm 1}, Y^{\pm 1}]$, which is isomorphic to $\mathbb F_p[X^{\pm 1}, Y^{\pm 1}]$. It satisfies
\begin{equation}\label{eq-exseq <a,b>}
\exseq{\mathbb F_p[X^{\pm 1}, Y^{\pm 1} ]}{\langle a, b \rangle}{\Z^2}.
\end{equation}
Moreover, it is a quotient of the free $p$-metabelian group of rank $2 :$ there exists $$p: B_2^{(p)} \twoheadrightarrow \langle a, b \rangle.$$ 
This $p$ is a morphism of extensions (\ref{eq-exseq B_2^p}) to (\ref{eq-exseq <a,b>}), hence an isomorphism.
\end{proof}

\begin{rem}
If the extension splits, the group $H$ has a subgroup isomorphic to the bigger group $(\Z/p\Z) \wr \Z^2$.
\end{rem}

We may now complete the proof of Proposition \ref{prop-special subgroups}.

\begin{proof}[Proof of Proposition \ref{prop-special subgroups}]
Let $G$ be a finitely generated metabelian group. Again, we can reduce by Proposition \ref{prop-fin ind} to the case of $G$ fitting in an exact sequence such as (\ref{eq-exseq trans mod}) with torsion-free quotient $\Z^d$, for some $d$ . Assume that $\krull(G) \geq 2$. There are two cases: 
\begin{enumerate}
\item $\krull^0(G) \geq 2$. Then, Corollary \ref{cor-ZZ} implies that $G$ has a subgroup isomorphic to $\Z \wr \Z$.
\item $\krull^t(G) \geq 2$. Let
\begin{equation*}
\exseq{M}{G}{Q}
\end{equation*} with $M$ abelian, and $Q$ torsion-free abelian of finite rank.
Then, $M$ has a submodule isomorphic to $\mathbb F_p[X^{\pm 1}, Y^{\pm 1}]$, for some prime $p$ (see Proposition \ref{prop-max family trans e}). The elements $X$ and $Y$ come from monomials of $Q \simeq \Z^d$. The subgroup generated by the ring $\mathbb F_p[X^{\pm 1}, Y^{\pm 1}]$ and lifts of the two corresponding monomials is an extension of $\mathbb F_p[X^{\pm 1}, Y^{\pm 1}]$ by $\Z^2$, with the usual action of $\Z^2$. We conclude with an application of Lemma \ref{lem-non split extension of Z2}.
\end{enumerate}

Consequences in term of return probability follows from computations for $\Z\wr\Z$ (see \cite{Erschler06}) and $B_2^{(p)}$ (see Proposition \ref{prop-return proba B_2^p}).
\end{proof}

\section{Return probability lower bounds via the construction of sequences of F\o lner couples}\label{sec-folner}

\subsection{F\o lner couples and return probability}

\begin{defn}[see \cite{CGP01} and \cite{Erschler06}] 
Let $G$ be a finitely generated metabelian group. We denote by $S$ a finite and symmetric generating set of $G$, and consider the associated word distance. Let $\mathcal{V}$ be a positive continuous increasing function on $[1, + \infty)$ whose inverse is defined on $[\mathcal{V}(1), +\infty)$. We say that $G$ admits \emph{a sequence of F\o lner couples adapted to $\mathcal{V}$} if there exists a sequence $(\OO_m, \OO_m')_{m \in \N}$ of pairs of non-empty finite sets $\OO_m' \subset \OO_m$ in $G$, with $\#\OO_m \nearrow \infty$ such that
\begin{enumerate}
\item $\#\OO_m'$ is a positive proportion of $\#\OO_m: \#\OO_m' \geq c_0 \#\OO_m$.
\item $\OO_m'$ lies linearly inside $\#\OO_m: \OO'_m S^m \subset \OO_m$.
\item $\mathcal{V}$ controls the size: $\#\OO_m \leq \mathcal{V}(m)$.
\end{enumerate}
\end{defn}

Given this, one can deduce a lower bound for the return probability, depending on $\mathcal V$ (see Coulhon, Grigor'yan and Pittet \cite{CGP01}, as well as Erschler \cite{Erschler06} for more general statements). We only recall the corollary that we need.

\begin{cor}[\cite{CGP01}]\label{cor-lower bound}
If a group $G$ admits a sequence of F\o lner couples adapted to a function of the form $\mathcal{V}(t) = C\exp(Ct^d)$, then 
\begin{equation*}
p_{2n}^G \succsim \exp(-n^\frac{d}{d+2}).
\end{equation*}
\end{cor}

\subsection{F\o lner couples for split metabelian groups}

In this part, we prove Theorem \ref{thm-lower bound torsion}. For any finitely generated metabelian group, Proposition \ref{prop-fin ind} provides us with a finite index subgroup $G$, of the same Krull dimension, admitting an exact sequence
\begin{equation}\label{eq-exseq 5}
\exseq{M}{G}{\Z^d},
\end{equation}
with $M$ abelian. The return probability of $G$ is equivalent to that of the initial group. 
Morevover, as the return probability increases when going to a subgroup, Theorem \ref{thm-embedding} imply that, when looking for lower bounds on the return probability for metabelian groups, we can reduce to the split case.

Therefore, it will be enough to prove Theorem  \ref{thm-lower bound torsion} in the split case, namely for a finitely generated metabelian group $G$ of the form $G = M \rtimes \Z^d$, with $M$ abelian. Our goal is then to produce sequences of F\o lner couples for such a group $G$ so as to apply Corollary \ref{cor-lower bound}. These sequences will take the following form.

\begin{defn}
Let $G = A \rtimes B$ be a finitely generated semi-direct product. Equip $G$ with a generating set $S = S_A \sqcup S_B$ where $S_A \subset A$ generates $A$ as a $\Z B$-module and $S_B \subset B$ generates $B$.
We should say that $G$ admits a sequence of \emph{split} F\o lner couples if it has a sequence of F\o lner couples of the form $(\Omega_m \times F_m, \Omega'_m \times F'_m)_m$, with $\Omega_m, \Omega'_m \subset A$ and $F_m, F'_m \subset B$. 
\end{defn}

\begin{rem}
In the previous definition, the sequence $(F_m, F_m')_m$ is as well a sequence of F\o lner couples for $B$.
\end{rem}

\subsubsection{When $M$ is a ring of positive prime characteristic}\label{ssec-ring}~

We deal first with the case of $M$ being a ring $\Z \Z^d/ \PP$ of positive characteristic, with $\PP$ a prime ideal. 
Hence, in this part, $M$ is $\mathbb F_p[X_1^{\pm 1}, \dots, X_d^{\pm 1}]/\mathcal{P}$, for some prime $p$ and some prime ideal $\PP$. Denote by $\pi$ the canonical projection
\begin{equation*}
\pi: \mathbb F_p[X_1^{\pm 1}, \dots, X_d^{\pm 1}] \rightarrow \mathbb F_p[X_1^{\pm 1}, \dots, X_d^{\pm 1}]/\mathcal{P}.
\end{equation*}
and by $B_m$ the polynomials of degree bounded by $m$, that is
\begin{equation}\label{eq-Bm}
B_m = \{ P \in \mathbb F_p[X_1^{\pm 1}, \dots, X_d^{\pm 1}] \mid \supp(P) \subset \llbracket-m, m\rrbracket^d \}.
\end{equation}
Note that $B_m$, as well as $\pi(B_m)$ are abelian groups. 

\medskip

The following proposition is straightforward.

\begin{prop}\label{prop-folner seq ring p}
The sequence 
\begin{equation*}
(\OO_m, \OO_m' )_m = ( \pi(B_{2m}) \rtimes \llbracket-2m, 2m\rrbracket^d, \pi(B_{2m}) \rtimes \llbracket-m, m\rrbracket^d)_m
\end{equation*}
is a sequence of split F\o lner couples adapted to $\mathcal{V}(m) = \# B_{2m} \leq C\exp(Cm^d)$.
\end{prop}

\begin{rems}
Obviously, this $\mathcal{V}$ is not optimal, although it gives, by Corollary \ref{cor-lower bound}, a lower bound on the return probability depending only on the rank of $G_{ab}$: 
\begin{equation*}
p_{2n}^G \succsim \exp(-n^\frac{d}{d+2}).
\end{equation*}

The following lemma aims to improve it, making use of Noether's normalization theorem \ref{thm-noether} to control the size.
\end{rems}

\begin{lem}\label{lem-folner seq for ring p}
There exists a constant $C$ such that
\begin{equation}\label{eq-pi(B_m) dim k}
\# \pi(B_m) \leq C\exp(Cm^k),
\end{equation}
where $k$ is the Krull dimension of $\mathbb F_p[X_1^{\pm 1}, \dots, X_d^{\pm 1}]/\mathcal{P}$.
\end{lem}

\begin{proof}
By Noether's normalization theorem \ref{thm-noether}, there exists elements $z_1, \dots, z_k$ in ${\mathbb F_p[X_1^{\pm 1}, \dots, X_d^{\pm 1}]/\mathcal{P}}$ giving it the structure of a finitely generated module over its subring $\mathbb F_p[z_1, \dots, z_k]$. Hence, each element $X_i^{\pm 1}$ is integral over $\mathbb F_p[z_1, \dots, z_k]$ and is a root of a monic polynomial with coefficients in $\mathbb F_p[z_1, \dots, z_k]$. We may assume that these polynomials have the same degree $D$. Let $N$ be such that the support of all coefficients appearing in these polynomials lies in $\llbracket-N, N\rrbracket^k$.

Then, for $X$ being one of the $X_i^{\pm 1}$, we may write each power $X^{D + r}$ as a polynomial in $X, \dots X^{D-1}$, whose coefficients have support inside $\llbracket-rN, rN\rrbracket^k$.

Therefore, we can write any element of $\pi(B_m)$ as a linear combination of monomials $X_1^{n_1}, \dots X_d^{n_d}$ with $(n_1, \dots, n_d)\in [-D, D]^d$ and coefficients which are elements of $\mathbb F_p[z_1, \dots, z_k]$, with support in $\llbracket-Cm, Cm\rrbracket^k$, for some constant $C$.
\end{proof}

\subsubsection{F\o lner sequences for extensions}

In this part, we recall and prove Proposition \ref{prop-split torsion} stated in the introduction.

\begin{prop3}\label{prop-torsion split}
Let $G$ be a finitely generated split metabelian group, whose derived subgroup is torsion. Assume $G$ has Krull dimension $k \geq 1$. Then
\begin{equation*}
p_{2n} \sim \exp(-n^\frac{k}{k+2}).
\end{equation*}
\end{prop3}

\begin{proof}
As mentionned at the beginning of this subsection, we may assume that $G$ fits inside an exact sequence such as \ref{eq-exseq 5}, with torsion-free quotient.

The upper bound comes from Proposition \ref{prop-split case}. We are left with the lower bound: it is enough (see Corollary \ref{cor-lower bound}) to show that $G$ admits a sequence of F\o lner couples adapted to $\mathcal V(m) = C\exp(Cm^k)$, for some $C$.

The proof of this fact uses recurrence along the decomposition \ref{eq-dec M} of the subgroup $M$ as an increasing sequence of submodules. Initialization is given by Proposition \ref{prop-folner seq ring p} and Lemma \ref{lem-folner seq for ring p}, and iteration follows from the next lemma.
\end{proof}

\begin{lem}\label{lem-remonte fp t}
Let $G = M \rtimes \Z^d$ be a split metabelian group, with $M$ a torsion $\Z\Z^d$-module of Krull dimension at most $k$.

Assume that
\begin{enumerate}
\item there exists a submodule $M_1$ in $M$ such that $M_1 \rtimes \Z^d$ admits a split sequence of F\o lner couples  $(\OO_m \rtimes \llbracket-2m, 2m\rrbracket^d, \OO_m' \rtimes \llbracket-m, m \rrbracket^d)_m$ adapted to $\mathcal{V}(m) =  C\exp(Cm^{k})$, where $\OO_m$ and $\OO_m'$ are abelian subgroups of $M_1$, 
\item $M/M_1$ is a ring of the form $\Z(\Z^d)/ \mathcal{P}$, with $\PP$ a prime ideal.
\end{enumerate}

Then, the group $G$ admits a split sequence of F\o lner couples adapted to a function $\mathcal{V'}(m) = C'\exp(C'm^k)$, of the form $(\GGG_m \rtimes \llbracket-2m, 2m\rrbracket^d, \GGG_m' \rtimes \llbracket-m, m \rrbracket^d)_m$. Moreover, the projections $\GGG_m, \GGG'_m$ on $M$ are abelian subgroups.
\end{lem}

\begin{proof}
We have $M/M_1 = \Z\Z^d/ \PP$ for some prime ideal $\PP$, and this module has dimension at most $k$. Its characteristic is a prime $p$. 
Let $1_Q$ be a lift of the unit of $M/M_1$ in $M$. The $\Z\Z^d$-module $M_2$ generated by $1_Q$ surjects onto $M/M_1$, and is of the form $\Z\Z^d / I$, for some ideal $I \subset \PP$. Because $M$ is a torsion module, $M_2$ is actually a ring of characteristic $n$, for some multiple $n$ of $p$. 

We cannot apply directly Noether's normalization theorem to $M_2$. To get around this, note that the projection $M_2 \twoheadrightarrow M/M_1$ have finite fibres of cardinality $\frac{n}{p}$. Hence, we can do as in part \ref{ssec-ring}: consider the finite subgroup $B_m$ in $\Z\Z^d$ (defined in \ref{eq-Bm}). Noether's normalization theorem provides us with an upper bound on the cardinality of its projection  $\pi(B_m)$ into $M/M_1$, where $\pi: \Z\Z^d \twoheadrightarrow M/M_1$. Let $\LLL_m$ be the pullback of $\pi(B_{2m})$ in $M_2$. It is also the projection of $B_{2m}$ in $M_2 = \Z\Z^d/I$. Because the fibres are finite, the cardinality of $\LLL_m$ satisfies \ref{eq-pi(B_m) dim k}, for some constant $C_1$ :
\begin{equation*}
\# \LLL_m \leq C_1\exp(C_1m^k) = \mathcal{V}_1(m).
\end{equation*}
Note that $\LLL_m$ is an abelian subgroup of $M$.

We fix a generating set $S$ for $G$ as follows: $$S = S_1 \cup \{(\pm \mathrm 1_Q, 0) \},$$ such that $S_1 = S_1' \cup \{ (0, \pm e_1), \dots (0, \pm e_d) \}$ with $S_1'$ a finite and symmetric generating set for the finitely generated $\Z(\Z^d)$-module $M_1$ and $(e_i)_{i=1}^d$ being the canonical basis of $\Z^d$. We choose $S_1$ so as to be compatible with the sequence of F\o lner couples of $M_1 \rtimes \Z^d$.

Consider the following sequence of couples indexed by $m \in \mathbb N$ :
\begin{equation*}
F_m = \{ P.\mathrm 1_Q  + \omega \mid P \in \LLL_m, \omega \in \OO_m  \} \rtimes \llbracket-2m, 2m\rrbracket^d
\end{equation*}
\begin{equation*}
F_m' = \{ P.\mathrm 1_Q  + \omega \mid P \in \LLL_m, \omega \in \OO_m'  \}\rtimes \llbracket-m, m\rrbracket^d .
\end{equation*}
We claim that this is a sequence of F\o lner couples for the group $G$, adapted to  $\mathcal{V'}(m) = C'\exp(C'm^k)$, for some constant $C'$. Note that the projections on $M$ are again abelian subgroups.
Indeed, $\# F_m \leq \mathcal V_1(m)\#\OO_m(4m + 1)^d \leq C'\exp(C'm^k)$. The choice of a sequence of F\o lner couples for $\OO_m$ and $\OO_m'$ implies $\# \OO_m' \geq c \# \OO_m,$ and that $F_m' S^m \subset F_m$.

The last condition, namely  $\frac{\# F_m'}{\# F_m} \geq c$, follows from the following straightforward fact, applied to $H = M$, $\Omega = \Omega_m, \Omega' = \Omega_m$ and  $X = \Lambda_m$.

\begin{fac}
Let $H$ be a group and $\Omega'\subset \Omega$ be two subgroups of $H$. Assume that $X$ is a subset of $H$ so that $\Omega X$ and $\Omega' X$ are subgroups. Then
\begin{align*}
[\Omega X : \Omega' X] \leq [\Omega : \Omega'].
\end{align*}
\end{fac}
\end{proof}

We may now complete the proof of Theorem \ref{thm-lower bound torsion}.

\begin{proof}[Proof of Theorem \ref{thm-lower bound torsion}]
As the return probability increases when taking subgroups, the theorem follows from Proposition \ref{prop-torsion split} and Theorem \ref{thm-embedding}. 
\end{proof}

\subsection{Proof of Theorem \ref{thm-characterization}}

The proof of Theorem \ref{thm-characterization} requires to manage with the torsion submodule and torsion-free quotient of $M$ so as to deal with both cases. Therefore, Lemma \ref{lem-remonte fp t}, that provides F\o lner couples for extension in the torsion case, will not suffice.

We first prove a proposition stating that admitting F\o lner couples descends to the quotient in finitely generated groups. The fact that it goes to a subgroup is due to Erschler \cite{Erschler06}. 

\begin{prop}\label{prop-folner quotient}
Let $G$ be a fintely generated group and $H$ be a quotient of $G$. 

Assume that $G$ admits a sequence of F\o lner couples adapted to a function $\mathcal V$. Then, so does $Q$. 
\end{prop}

\begin{proof}
Equip $G$ with a finite and symmetric generating set $S$ and $H$ with the projection $T$ of $S$ to the quotient.

Denote by $(\Omega_n, \Omega'_n)_n$ the given sequence of F\o lner couples. If $h$ is a function from a discrete group $K$ to $\R$,  $\lVert h \rVert_1 = \sum_{k\in K}\lvert h(k)\rvert$ designates its $L^1$-norm.

Set $f_n = 1_{\Omega'_n}$ and $\lvert\nabla_nf_n\rvert(g) = \sup_{x\in S^n} \lvert f_n(g) -f_n(gx)\rvert$. The latter gradient measures the variations of $f_n$ along $S^n$.
We have : $$\lVert \lvert\nabla_nf_n\rvert \rVert_1 \leq \lVert f_n \rVert_1.$$
Let $\widehat{f_n}$ be the function from $Q$ to $\R$ defined by
\begin{align*}
\widehat{f_n}(q) = \sum_{g \in \pi^{-1}(q)} f(g).
\end{align*}
The norm is preserved :  $\lVert \widehat{f_n} \rVert_1 = \lVert f_n \rVert_1 = \# \Omega'_n $.
Similarly, consider the gradient $\lvert\widehat \nabla_n\widehat f_n\rvert(q) = \sup_{x\in T^n} \lvert \widehat f_n(q) -\widehat f_n(qx)\rvert$.
Then, we have : $\lVert \lvert\widehat \nabla_n\widehat f_n\rvert \rVert_1 \leq \lVert \lvert\nabla_nf_n\rvert \rVert_1 \leq  \lVert \widehat f_n \rVert_1$.

For a subset $A$ of $G$ denote by $\partial_n A$ the $n$-boundary of $A$, that is : $\partial_n A = AS^n \cap A^C S^n$. The following version of the co-area formula (see for instance \cite{T08} (3.1) for a proof): 
\[ \frac{1}{2}\int_{\R^+} \# \partial_n\{f > t\} \mathrm{dt} \le \lVert \nabla_n f \rVert_1 \le \int_{\R^+} \# \partial_n\{f > t\} \mathrm{dt}, \]
implies that there exist $t \ge 0$ and $C > 0$, such that 
\begin{align*}
\#  (\partial_n \{ \widehat{f_n} > t \})  \leq C \# \{ \widehat{f_n} > t \}.
\end{align*}
Set $\widehat{\Omega'_n} = \{ \widehat{f_n} > t \}$ and $\widehat{\Omega_n} = \widehat{\Omega'_n} \cup \partial_n \widehat\Omega'_n$. We claim that $(\widehat{\Omega_n}, \widehat{\Omega'_n})$ is a sequence of F\o lner couples for $Q$ adapted to $\mathcal V$.
\end{proof}

\begin{rem}
It is possible to elaborate on the technique appearing in this proof and to generalize this lemma to the context of locally compact compactly generated groups. The good setting for this seems to be the theory of isoperimetric profiles of groups through an approach similar to \cite{T08, T13}. This will be the content of a forthcoming paper \cite{J17}.
\end{rem}

\begin{cor}\label{cor-split fp quotient}
Let $G = M \rtimes Q$ be a finitely generated semi-direct product and $H = M'\rtimes Q$, with $M'$ a quotient of $M$ as a $\Z Q$-module.

Assume $G$ admits a sequence of split F\o lner couples of exponential size whose projection onto $Q$ is $(F_n, F_n')_n$, then  $H$ admits a sequence of split F\o lner couples of exponential size that projects onto $(F'_nS_Q^n, F_n')_n$.
\end{cor}

\begin{proof}
Write $(\OO_n, \OO'_n) = (A_n \times F_n, A'_n \times F'_n)_n$ for the sequence of split F\o lner couples of $G$ and denote by $\pi$ the projection of the $\Z Q$-module $M$ onto $M'$. Coming back to the proof of the proposition above, a computation gives :  for any $(m', q) \in H$, $\widehat{f_n}(m', q) = \# (A_n' \cap \pi^{-1}(m'))$. Hence, $\widehat{\OO'_n} = \widehat{A'_n} \times F_n'$ for some $\widehat{A'_n} \subset \pi(A'_n)$ and by construction the projection onto $Q$ of $\widehat{\OO_n}$ is $F_n'S_Q^n$.
\end{proof}

The next lemma will allow to combine F\o lner couples along the decompostion of a finitely generated module. 

\begin{lem}\label{lem-lifting}
Let $Q$ be a finitely generated abelian group and $M$ be a finitely generated $\Z Q$-module. Let $S_Q$ be a finite generating set for $Q$ and $S_M$ be a finite generating set for the $\Z Q$-module $M$. Assume that
\begin{enumerate}
\item $M$ has  submodules $M_1$ and $M_2$ so that $M_2$ is a cyclic $\Z Q$-module and its projection onto $M/M_1$ is the whole of $M/M_1$.
\item The group $M_1 \rtimes Q$ admits a sequence of F\o lner couples of exponential size that projects onto a sequence $(F_n, F'_n)_n$ in $Q$.
\item The group  $M_2 \rtimes Q$ admits a sequence of split F\o lner couples of exponential size that  projects onto a sequence $(G_n, F'_n)_n$ in $Q$ with $F_n \subset G_n$ for any $n$.
\end{enumerate}
Then, the group $M \rtimes Q$ admits a sequence of F\o lner couples of exponential size that projects onto the  sequence $(F'_n S_Q^n, F'_n)_n$ in $Q$.
\end{lem}

\begin{proof}
The module $M_1\times M_2$ surjects onto $M$, hence the group $H = (M_1 \times M_2)\rtimes Q$, where $Q$ acts diagonally, surjects onto $M \rtimes Q$. By Corollary \ref{cor-split fp quotient}, it is enough to construct F\o lner pairs for $H$, whose projection onto $Q$ is $(F_n, F'_n)_n$.

Write $(\Omega_n, \Omega'_n)_n$, resp. $(\Lambda_{n} \times G_n, \Lambda'_{n} \times F'_n)_n$ the sequence of F\o lner couples of $M_1 \rtimes Q$, resp. $M_2\rtimes Q$, and set for all $n$
\begin{align*}
\Delta_n  = \Omega_n \times \Lambda_n,~~~~~~~~ \Delta'_n  = \Omega'_n \times \Lambda'_n. 
\end{align*}

The sequence $(\Delta_n, \Delta'_n)_n$ is a sequence of F\o lner couples of exponential size, that projects onto the sequence $(F_n, F'_n)_n$ in $Q$. 
\end{proof}

The proof of Theorem \ref{thm-characterization} makes use the following ingredient of the proof of the main theorem in Pittet and Saloff-Coste \cite{PSC03}.

\begin{prop}[\cite{PSC03}, Proposition $7.8$]\label{thm-psc}
Let $\Gamma$ be a finitely generated torsion-free soluble group of finite Prüfer rank satisfying
\[ \exseq{N}{\Gamma}{A}, \]
with $N$ nilpotent and torsion-free and $A$ free abelian of finite rank $d$.

Then, $\Gamma$ admits a sequence of F\o lner couples of exponential size. If the extension splits, so does the F\o lner sequence. Moreover, the projection onto $A$ of this sequence  is the classical F\o lner sequence $ (\llbracket -2n, 2n \rrbracket^d, \llbracket -n, n \rrbracket^d)_n$ of F\o lner couples of $A$.
\end{prop}

\begin{proof}[Proof of Theorem \ref{thm-characterization}]
The reverse implication follows from Proposition \ref{prop-special subgroups}.

To prove the direct implication, let $G$ be a finitely generated metabelian group of Krull dimension at most $1$. We wish to prove a lower bound on its return probability, therefore, by Proposition \ref{prop-fin ind} and Theorem \ref{thm-embedding}, we may assume that $G$ is a semi-direct product $M \rtimes Q$ where $M$ is abelian and $Q$ is free abelian. Equip $G$ with a generating set $S = S_M \sqcup S_Q$ where $S_M\subset M$ generates $M$ as a $\Z Q$-module and $S_Q\subset Q$ generates the group $Q$. Denote by $d$ the rank of $Q$ and set $F = (F_n, F'_n)_n = (\llbracket -2n, 2n \rrbracket^d, \llbracket -n, n \rrbracket^d)_n$ the classical sequence of F\o lner couples for $Q$.

Proposition \ref{prop-dec M} provides us with submodules of $M$
\[ \{0\} = M_0 \leq M_1 \leq \dots \leq M_n = M\]
such that $M_{i+1}/M_i$ is isomorphic to $\Z Q/ \mathcal{P}_i$ with $\mathcal{P}_i$ a prime ideal of $\Z Q$. Each of these cyclic modules $\Z Q/ \mathcal{P}_i$ has dimension at most $1$ and by Proposition \ref{prop-folner seq ring p} and Lemma \ref{lem-folner seq for ring p} for the torsion case, and Propositions \ref{prop-dim 1 torsion-free} and \ref{thm-psc} for the torsion-free case, the groups $\Z Q/ \mathcal{P}_i \rtimes Q$ admits F\o lner couples of exponential size whose projection onto $Q$ is $F$. 

We combine them by iterative applications of Lemma \ref{lem-lifting} to obtain a sequence of F\o lner couples  of exponential size in $G$. The conclusion then follows from \ref{cor-lower bound}.
\end{proof}

\subsection{Application to the $L^2$-isoperimetric profile} (see for instance \cite{Cou00})  The \emph{$L^2$-isoperimetric profile} of a finitely generated group $G$, equipped with a finite and symmetric generating set $S$, is the non-decreasing function
\[ j_G(v) = \sup_{\#A \le v}~ \sup_{f \in L^2(A)}~ \frac{\lVert f \rVert_2}{\sup_{s\in S} \lVert f - f(s^{-1}.) \rVert_2 }, \]
where $L^2(A)$ denotes the set of functions $f : G \rightarrow \R$ supported in $A$.

We recall the following theorem from \cite{T13}, which is a particular case of\cite{Cou00} Theorem $6.1$, generalized to the setting of metric measure spaces in \cite{T08}.

\begin{thm}[\cite{T13}, Theorem $2.7$]
Let $G$ be a finitely generated group. Then the isoperimetric profile of $G$ satisfies $j_2(v) \sim \ln v$ if and only if $p_{2n}^G \sim \exp(-n^\frac1{3})$.
\end{thm}

Hence, our main theorem may be stated as follows:

\begin{cor}
Let $G$ be a finitely generated metabelian group of exponential growth. Let $d$ be the Krull dimension of $G$. Then,
\[ j_G(v) \sim \ln v \Leftrightarrow d\le 1. \]
\end{cor}

\bibliographystyle{alpha}
\bibliography{biblio_metabelian}

\end{document}